\documentclass[10pt,reqno]{amsart}
\usepackage{amssymb,amsmath,amscd,mathrsfs}
\usepackage[english]{babel}

\newtheorem{theorem}{Theorem}[section]
\newtheorem{lemma}[theorem]{Lemma}
\newtheorem{corollary}[theorem]{Corollary}
\newtheorem{proposition}[theorem]{Proposition}

\theoremstyle{definition}
\newtheorem{definition}[theorem]{Definition}

\begin{document}


\title[$\Phi$-Harmonic Functions and $\ell^\Phi$-Cohomology]
{$\Phi$-Harmonic Functions on Discrete Groups \\ and the First $\ell^\Phi$-Cohomology}
\author{Yaroslav Kopylov}
\address{Sobolev Institute of Mathematics, Pr. Akad. Koptyuga 4,
630090, Novosibirsk, Russia and  Novosibirsk State University,
ul.~Pirogova 2, 630090, Novosibirsk, Russia }
\email{yakop@math.nsc.ru}
\author{Roman Panenko}
\address{Sobolev Institute of Mathematics, Pr. Akad. Koptyuga 4,
630090, Novosibirsk, Russia}
\email{panenkora@gmail.com}
\thanks{The first-named author was partially supported by the Russian Foundation for Basic Research 
(Grant~12-01-00873-a). Both authors were supported by the State Maintenance Program for the Leading 
Scientific Schools and Junior Scientists of the Russian Federation (Grants~NSh-921.2012.1, 
NSh-2263.2014.1)}

\date{}

\begin{abstract}
We study the first cohomology groups of a~countable discrete group~$G$ with coefficients
in a $G$-module $\ell^\Phi(G)$, where $\Phi$ is an $N$-function of class
$\Delta_2(0)\cap \nabla_2(0)$. In development of ideas of Puls and
Martin~---Valette, for a~finitely generated group~$G$, we introduce the discrete
$\Phi$-Laplacian and prove a theorem on the decomposition of the~space of $\Phi$-Dirichlet
finite functions into the direct sum of the spaces of $\Phi$-harmonic functions and
$\ell^\Phi(G)$ (with an appropriate factorization). We also prove that if a~finitely
generated group~$G$ has a~finitely generated infinite amenable subgroup with infinite
centralizer then $\overline{H}^{1}(G,\ell^{\Phi}(G)) = 0$. In conclusion, we show
the triviality of the first cohomology group for a~wreath product of two groups
one of which is nonamenable.
\end{abstract}

\maketitle

\section*{Introduction}

In the recent decades, the methods of harmonic analysis on discrete structures have been
widely used for the study of these structures as well as of manifolds and topological
groups. In particular, the study of amenability and Kazhdan's property~(T) for
a~locally compact group is reduced to checking these properties for any lattice
of the group (see, for example, \cite{Lub}). Many problems in the geometry of a~Riemannian
manifold satisfying the~doubling condition and Poincar\'e's inequality are reduced
to the study of some graph, called the {\it discretization} of the manifold~\cite{CoulSC95}.

The questions of the ``growth'' of  discrete groups ``at infinity'' and the~presence
of fixed points in a~continuous isometric affine action of a~group on a~Banach space
are closely related to the group 1-cohomology with coefficients in the~corresponding
Banach module (see, for example, \cite{BeHaV08, Puls06}).

The present article deals with sufficient conditions for the~triviality of the first
cohomology group $H^1(G,\ell^\Phi(G))$ of a~countable discrete group~$G$ for
a $\Delta_2$-regular $N$-function~$\Phi$. It is worth noting that earlier, 
in~\cite{LeMo09, LeMoMo10}, the $\Phi$-Laplacian was already considered on bounded domains 
in~$\Bbb R^n$, and Orlicz spaces in the context of~discrete groups were examined 
by Kami{\'n}ska and Musielak in~\cite{KamMus89}.

The article has the following structure:

In Sec.\,\ref{orl-s}, we give some basic notions concerning $N$-functions and the theory of discrete
Orlicz spaces. In Sec.\,\ref{1-cohom}, we recall the definitions of the 1-cohomology $H^1(G,V)$ and
the~reduced 1-cohomology $\overline{H}^1(G,V)$ of a~topological group~$G$ with coefficients
in a~Banach $G$-module $V$, expose a~number of~known assertions, and obtain a~sufficient
condition for the~triviality of the~cohomology $H^1(G,\ell^\Phi(G))$ for
a~$\Delta_2$-regular $N$-function~$\Phi$ in~terms of the~triviality of the 1-cohomology
of its nonamenable subgroup with some additional assumptions (Theorem~\ref{inf-norm}). 
In~Sec.~\ref{phi-harm},
by~analogy with the $p$-Laplacian and $p$-harmonic functions on finitely generated groups,
considered by Puls in~{Puls06} and Martin~--- Valette in~\cite{MV07} , we introduce the~notions
of the $\Phi$-Laplacian and a~$\Phi$-harmonic function for an~arbitrary $N$-function
$\Phi$ on a~finitely generated group. The~main result of~Sec.\,\ref{phi-harm} 
is Theorem~\ref{decomp} on the~decomposition of the~space of $\Phi$-Dirichlet finite functions, where
$\Phi$ is a~continuously differentiable strictly convex $N$-function belonging
to $\Delta_2(0)\cap\nabla_2(0)$ into the~sum of the space of $\Phi$-harmonic functions
and the~closure of~$\ell^\Phi(G)$. In~Sec\,\ref{central}, 
we use the results of~Sec.\,\ref{phi-harm} to establish
sufficient conditions for the~triviality of the~reduced 1-cohomology
$\overline{H}^1(G,\ell^\Phi(G))$ of a~finitely generated infinite group and prove a~condition
for the triviality of the 1-cohomology of a~wreath product of groups. We obtain analogs
of Theorems~4.2, 4.3, and~4.6 in~\cite{MV07} for the~case of $\ell^\Phi$-cohomology
(Theorem~\ref{large_centr}, Corollary~\ref{inf-center}, Theorem~\ref{wreath}).

\section{Orlicz Spaces}\label{orl-s}

Here we recall some notions from the theory of Orlicz spaces which we will need in the sequel.
Below $X$ is a countable set endowed with a counting measure. 

%
%
\begin{definition}
A function $\Phi: \mathbb{R} \to \mathbb{R}$ is called an $N$\emph{-function} if it can 
be represented as 
$$
\Phi(x) = \int\limits_{0}^{|x|} \varphi(t)\, dt,
$$
where the function $\varphi(t)$ is defined for $t \geqslant 0$, non-decreasing, right-continuous,
$\varphi(t)>0$ if $t > 0$, $\varphi(0) = 0$ and  $\lim_{t \to \infty}\varphi(t) = \infty$. In what 
follows, $\Phi'$ stands for this function $\varphi$.

    An $N$-function $\Phi$ has the following properties: 
	\begin{itemize}
         	\item $\Phi(x)>0$ if $x > 0$;
		\item $\Phi$ is even and convex;
	        \item $\lim\limits_{x \to 0}\frac{\Phi(x)}{x} = 0$, 
$\lim\limits_{x \to \infty}\frac{\Phi(x)}{x} = +\infty$.
	\end{itemize}
\end{definition}
%
%
	\begin{definition} 
		If $\Phi$ is an $N$-function then the function given by 
$$
\Psi(x) = \int\limits_{0}^{x}(\Phi')^{-1}(t) \,dt, \quad \mbox{where}~ (\Phi')^{-1}(x) = \sup\limits_{\Phi'(t) \le x} t,
$$
	 is called \emph{complementary to $\Phi$}.
	\end{definition}

If $\Psi$ is the $N$-function complementary to an $N$-function $\Phi$ then $\Phi$ is complementary to~$\Psi$.  

		{\bf Remark. } For a pair $\Phi, \Psi$ of complementary $N$-functions, \emph{Young's inequality} 
$$
ab \le \Phi(a) + \Psi(b)
$$
holds for all nonnegative $a$ and $b$; it becomes an equality if and only if $b = \Phi'(a)$.
%
%

\begin{definition}
An $N$-function $\Phi$ is said to satisfy the {\em $\Delta_2$-condition for small~$x$},
which is written as $\Phi\in\Delta_2(0)$, if there exist constants $x_0>0$,
$K>2$ such that $\Phi(2x)\le K \Phi(x)$ for $0\le x \le x_0$; 
$\Phi$ satisfies the {\em $\nabla_2$-condition for small~$x$}, which is denoted symbolically 
as $\Phi\in\nabla_2(0)$ if there are constants $x_0>0$ and
$c>1$ such that $\Phi(x)\le \frac{1}{2c}\Phi(cx)$ for $0\le x\le x_0$.
\end{definition}

Below we will call $N$-functions $\Phi\in \Delta_2(0)$ ($\Phi\in \nabla_2(0)$) 
{\em $\Delta_2$-regular} ({\em $\nabla_2$-regular}). As is well known,
an $N$-function is $\nabla_2$-regular if and only if the complementary
$N$-function is $\Delta_2$-regular.

	
\begin{definition}
An function $\Phi$ is called {\em uniformly convex} if, given any $a \in (0,1)$,
 there exists $\beta(a) \in (0,1)$ such that  
$$
\Phi\left(\frac{u + bu}{2}\right)\le\frac{1}{2}(1 - \beta(a))(\Phi(u)-\Phi(bu))
$$
for any $b \in [0,1]$  and $u \geqslant  0$.
\end{definition}

%
\begin{definition}	
The Orlicz class $\tilde\ell^{\Phi}(X)$ is the set of real-valued functions on~$X$ for which 
$$
\rho_\Phi(f) := \sum_{x \in X}{\Phi(f(x))}<\infty.
$$
\end{definition}

We will use the notation
$$
\tilde\ell^{\Phi}_1(X) =\left\{ f \in \tilde\ell^{\Phi}(X)~
{\Bigl | \Bigr.}~\sum_{x \in X}{\Phi(f(x))} \le 1 \right\}$$

%

\begin{definition}
The linear space 
$$
\ell^\Phi(X) = \{ f:X\to\mathbb{R} \,:\, \rho_\Phi(af)<\infty ~\text{for {\it some} $a>0$}\} 
$$
is called an {\it Orlicz space} on~$X$. 
\end{definition}

{\bf Remark.} As is well known (see, for example, \cite{RaRen02}), 
$\tilde\ell^{\Phi}(X)$ is a~linear space (and thus coincides with $\ell^\Phi(X)$) 
if and only if $\Phi \in \Delta_2(0)$.


\begin{definition}
If $f\in \ell^\Phi(X)$ then the {\it Orlicz norm} of~$f$ is, by definition,  
$$
\|f\|_\Phi:=\|f\|_{\ell^{\Phi}(X)} := \sup_{u \in \tilde\ell^{\Psi}_{1}}
\left|\sum_{x \in X}f(x)u(x)\right|.
$$
\end{definition}

\begin{definition}
The {\it gauge} (or {\it Luxemburg}) {\it norm} of a function $f\in \ell^\Phi(X)$
is defined by the formula
$$
\|f\|_{(\Phi)}:= \|f\|_{\ell^{(\Phi)}(X)} :=
\inf \biggl\{ k>0 \, : \, \rho_\Phi\biggl(\frac{f}{k}\biggr)\le 1 \biggr\}.
$$
\end{definition}

It is well known that the Orlicz and gauge norms are equivalent, namely (see, for example,
\cite[pp.~61--62, Proposition~4]{RaRen91}):
$$
\|f\|_{(\Phi)} \le \|f\|_\Phi  \le  2 \|f\|_{(\Phi)}.
$$

It what follows, unless otherwise specified, we tacitly endow the Orlicz space 
$\ell^\Phi(X)$ with the gauge norm $\|\cdot\|_{(\Phi)}$.


The following proposition is well known for Orlicz spaces on general measure spaces
(cf. \cite[Proposition~8, p.~79]{RaRen91}); we formulate and prove its ``countable'' 
version for the reader's convenience.
	
\begin{proposition}	\label{orl_class}
	 Let $\Phi \in \Delta_2(0)$ and let $\Psi$ be the complementary $N$-function to~$\Phi$.
If $f \in \tilde\ell^{\Phi}(X)$ then $\Phi'(|f|) \in \tilde\ell^{\Psi}(X)$.  
\end{proposition}

\begin{proof}
Let $f \in \ell^{\Phi}(X)$. Applying the equality in Young's inequality to $|f|$ 
and $\Phi'(|f|)$, we have 
\begin{equation}\label{young-iden}
\Phi(f) + \Psi(\Phi'(|f|)) = |f|\Phi'(|f|).
\end{equation}
Estimate $x\Phi(x)$ in a suitable neighborhood of zero. By definition, 
$\Phi(x) = \int\limits_{0}^{x} \varphi(t)\, dt $.  Further, 
$$
x\Phi'(x) = \int\limits_{x}^{2x} \varphi(x)\, dt \le  \int\limits_{x}^{2x} \varphi(t)\, dt 
\le  \int\limits_{0}^{2x} \varphi(t)\, dt = \Phi(2x)
$$
but we have $\Phi(2x) \le K\Phi(x)$ in view of the $ \Delta_2(0)$-condition. 
Consequently, we conclude $x\Phi'(x) \le K\Phi(x)$ in a suitable neighborhood of zero .
Involving~(\ref{young-iden}), we obtain $\Psi(\Phi'(|f|)) \le( K - 1)\Phi(|f|)$. 		
Then, choosing $\varepsilon$  such that $x\Phi'(x) \le K\Phi(x)$ for all  
$x \in [0, ~\varepsilon]$, we have  
	\begin{multline*}
\sum_{x \in X}\Psi(\Phi'(|f(x)|)) = \sum_{|f(x)| \le \varepsilon} 
\Psi(\Phi'(|f(x)|)) + \sum_{|f(x)| > \varepsilon}\Psi(\Phi'(|f(x)|))  \\
=\sum_{|f(x)| \le \varepsilon}\Psi(\Phi'(|f(x)|)) + C 
\le \sum_{|f(x)| \le \varepsilon}( K - 1)\Phi(|f(x)|) + C < \infty.	
	\end{multline*}
\end{proof}

%
\section{1-Cohomology}\label{1-cohom}

Let $G$ be a topological group and let~$V$ be a~topological $G$-module, i.e., a real or complex
topological vector space endowed with a~linear representation $\pi:G\times V\to V$, 
$(g,v)\mapsto \pi(g)v$. The space~$V$ is called a {\it Banach $G$-module} if $V$ is a~Banach 
space and $\pi$ is a~representation of~$G$ by isometries of~$V$. Introduce the notation:
$$
Z^1(G,V):=\{b:G\to V \text{~continuous~} \mid b(gh)=b(g) + \pi(g)b(h) \} \quad 
\text{({\it 1-cocycles})}; 
$$
$$
B^1(G,V) = \{ b\in Z^1(G,V) \mid (\exists v\in V)\, (\forall g\in G) \,\, b(g)= \pi(g) v - v \}
\quad \text{({\it 1-coboundaries})};
$$
$$
H^1(G,V) = Z^1(G,V)/B^1(G,V) \quad \text{({\it 1-cohomology with coefficients in~$V$})}.
$$

Endow $Z^1(G,V)$ with the topology of uniform convergence on compact subsets of~$G$ and 
denote by $\overline{B}^1(G,V)$ the closure of $B^1(G,V)$ in this topology. The quotient 
$\overline{H}^1(G,V) = Z^1(G,V)/\overline{B}^1(G,V)$ is called the 
{\it reduced 1-cohomology} of~$G$ with coefficients in the $G$-module~$V$. 

Let $V$ be a Banach $G$-module.
We say that $V$ \textit{has almost invariant vectors} or {\it almost has invariant vectors}
if, for every compact subset $F\subset G$ and every $\varepsilon>0$, there exists a~unit 
vector $v\in V$ such that $\|\pi(g)v-v\|\le\varepsilon$ for all $g\in F$.

Given a~closed normal subgroup~$N$ in~$G$ and a $G$-module $V$, the group $G$ acts on~$H^1(N,V|_N)$ 
as follows~(see~\cite{Bro82,Gui80}: On~$Z^1(N,V|_N)$, the action is defined by the~formula
$$
(g\cdot b)(n) = \pi(g)(b(g^{-1}ng)) 
$$
($b\in Z^1(N,V|_N)$, $g\in G$, $n\in N$). Since this action leaves $B^1(N,V|_N)$ invariant,
it defines an~action of~$G$ on~$H^1(N,V|_N)$. Since, for $m\in N$,
$$
(m\cdot b)(n) = b(n) + (\pi(n) b(m) - b(m)),
$$
the action of~$N$ on~$H^1(N,V|_N)$ is trivial, the action of~$G$ 
on~$H^1(N,V|_N)$ factors through~$G/N$. The following assertion is proved 
in~\cite[Corollary~6.4]{Bro82} and \cite[8.1]{Gui80}:

\begin{proposition}\label{rest} 
(1) There is an exact sequence 
\begin{equation*}
0 \rightarrow H^1(G/N,V^N) \overset{i_*}{\rightarrow} H^1(G,V) 
\overset{\mathrm{Rest}_G^N}{\longrightarrow} 
H^1(N,V|_N)^{G/N} \rightarrow \dots,
\end{equation*}
where $i:V^N\to V$ is the inclusion and $\mathrm{Rest}_G^N:H^1(G,V)\to H^1(N,V|_N)^{G/N}$ 
is the restriction map.

(2) If $V^N=0$ then $\mathrm{Rest}_G^N:H^1(G,V)\to H^1(N,V|_N)^{G/N}$ 
is an~isomorphism.
 
\end{proposition}

In~\cite{Kop13}, we obtained analogs of Corollary~2.4 in~\cite{MV07} 
and Proposition~2 in~\cite{BMV} in the context of locally compact groups. 
Here are the discrete versions of the results of~\cite{Kop13}:

\begin{proposition}\label{or-cohom} \cite[Proposition~2]{Kop13}
Suppose that $\Phi$ is an $N$-function of class~$\Delta_2(0)$.
If $G$ is a countable group then the following 
are equivalent:
           
{\rm(i)} $H^1(G,\ell^\Phi(G)) = \overline{H}^1(G,\ell^\Phi(G))$;

{\rm(ii)} $G$ is not amenable.
\end{proposition}

\begin{proposition}\label{amenab-subgr} \cite[Theorem~1]{Kop13}
Assume that $\Phi$ is an $N$-function of class~$\Delta_2(0)$.
Let $G$ be a countable group and let $H$ be an infinite subgroup in $G$. 
The following are equivalent:

{\rm(i)} The permutation representation of $H$ on $\ell^\Phi(G)$ has almost invariant 
vectors;

{\rm(ii)} $H$ is amenable.
\end{proposition}

For $H=G$ this was proved by Rao (see~\cite[Proposition~2, pp.~387--389]{Ra04a}).  

\begin{lemma}\label{cohom-rest}
Let $\Phi$ be an $N$-function of class $\Delta_2(0)$ and let $H$ be a subgroup of a countable 
discrete group $G$. Consider the properties

{\rm(i)} $\overline{H}^1(H,\ell^\Phi(H))=0$;

{\rm(ii)} $\overline{H}^1(H,\ell^\Phi(G)|_H)=0$;

{\rm(i$'$)} $H^1(H,\ell^\Phi(H))=0$;

{\rm(ii$'$)} $H^1(H,\ell^\Phi(G)|_H)=0$.

Then {\rm(i)}$\Longleftrightarrow${\rm(ii)} and {\rm(i$'$)}$\Longleftrightarrow${\rm(ii$'$)} 
\end{lemma}        

\begin{proof}
Make use of the~scheme of proof in~[8], where the $\ell^p$-case is discussed. Since
$\Phi$ is $\Delta_2$-regular, the~Orlicz space $\ell^\Phi(X)$ on a~countable set
consists of~those sequences $(x_n)$ for which $\sum\limits_{n=1}^\infty \Phi(x_n)<\infty$.
Thus, $\ell^\Phi(G)|_H$ can be identified with the $\ell^\Phi$-direct sum of $[G:H]$
copies of~$\ell^\Phi(H)$.

The implications {\rm(ii)}$\Longrightarrow${\rm(i)} and {\rm(ii$'$)}$\Longrightarrow${\rm(i$'$)} 
follow from the inclusions $H^1(H,\ell^\Phi(H)) \to H^1(H,\ell^\Phi(G)|_H)$ and 
$\overline{H}^1(H,\ell^\Phi(H)) \to \overline{H}^1(H,\ell^\Phi(G)|_H)$ induced by the mapping
$\iota: Z^1(H,\ell^\Phi(H)) \to Z^1(H,\ell^\Phi(G)|_H)$, $b \mapsto (b,0,0,\dots)$.

{\rm(i)}$\Longrightarrow${\rm(ii)} There is nothing to prove for $[G:H]<\infty$, and so 
we assume that $[G:H]=\infty$. Given a cocycle $b\in Z^1(H,\ell^\Phi(G)|_H)$, let 
$b_n\in Z^1(H,\ell^\Phi(H)$ be its $n$th ``component'', lying in $\ell^\Phi(H s_n)$. Take 
a~finite subset~$K$ in~$H$ and $\varepsilon>0$. 
It is well known (see \cite[p.~83, Theorem 12]{RaRen91} for details; the proof there 
is carried out  for global $\Delta_2$-regularity but is easily modified to 
the ``one-sided'' cases) that if an $N$-function $\Phi$ is $\Delta_2$-regular 
then
$$
\lim\limits_{\rho_\Phi(v)\to 0}\|v\|_{\ell^{(\Phi)}(G)}=0.
$$
This means that for every $\varepsilon>0$ there exists $\delta>0$ such that
$$
\rho_\Phi(f)<\delta \ \text{implies}\ \|f\|_{\ell^{(\Phi)}(G)}<\varepsilon.
$$
Choose $\delta_0$ such that if $\rho_\Phi(v)<\delta_0$ then
$\|v\|_{\ell^{(\Phi)}(G)}<\frac{\varepsilon}{2}$. Fix $\overline{N}>0$ such that
$\sum\limits_{n=\overline{N}+1}^\infty \Phi(b_n(h)) < \delta_0$ for all $h\in K$.
This means that, for the sequence 
$$
b_{> \overline{N}}(h)=(0,\dots,0,\Phi(b_{\overline{N}+1}(h)), \Phi(b_{\overline{N}+2})(h),\dots), \quad h\in K,
$$
we have $\rho_\Phi(b_{> \overline{N}}(h))< \delta_0$. Thus, by the choice of~$\delta_0$, 
$\|b_{>\overline{N}}(h)\|_{(\Phi)}<\frac{\varepsilon}{2}$. Now, for each $i=1,2,\dots,\overline{N}$,
find a function $v_i\in \ell^\Phi(H)$ with  
$\|b_i(h)-(\lambda_H(h)v_i-v_i)\|_{(\Phi)}<\varepsilon/2^{\overline{N}}$ for all $h\in K$.
Take $v=(v_1,\dots, v_{\overline{N}},0,0,\dots)\in \ell^\Phi(G)$. Then
$$
\|b(h) - (\lambda_G(h)v - v)\|_{(\Phi)} \le \sum_{i=1}^{\overline{N}}
\|b_i(h)-(\lambda_H(h)v_i-v_i)\|_{(\Phi)}
+\|b_{>\overline{N}}(h)\|_{(\Phi)} < \varepsilon
$$  
for each $h\in K$. Thus, $b\in \overline{B}^1(H,\ell^\Phi(G)|_H)$ and so 
$\overline{H}^1(H,\ell^\Phi(G)|_H)=0$. 

{\rm(i$'$)}$\Longrightarrow${\rm(ii$'$)} 
(a) If $H$ is finite then $H^1(H,\ell^\Phi(H))=H^1(H,\ell^\Phi(G)|_H)=0$.

(b) Suppose that $H$ is infinite. Then Proposition~\ref{or-cohom} implies that
$H$ is not amenable. Then from Proposition~\ref{amenab-subgr} it follows that 
$\ell^\Phi(G)|_H$ does not have almost invariant vectors. Involving 
Proposition~\ref{or-cohom}
and the implication {\rm(i)}$\Longrightarrow${\rm(ii)}, we obtain
$$
\overline{H}^1(H,\ell^\Phi(G)|_H)=H^1(H,\ell^\Phi(G)|_H)=H^1(H,\ell^\Phi(H))=0.
$$
\end{proof}

The following theorem is known for~$\ell^p$-cohomology 
(see \cite[Theorem~1]{BMV}): 

\begin{theorem}\label{inf-norm}
Suppose that $\Phi$ is an $N$-function of class~$\Delta_2(0)$. Let $N \le H \le G$ 
be a~chain of countable discrete groups such that $N$ is an~infinite 
normal subgroup in~$G$ and $H$ is nonamenable.
If $\overline{H}^1(H,\ell^\Phi(H))=0$ then $H^1(G,\ell^\Phi(G))=~0$. 
\end{theorem}

\begin{proof}
Since $H$ is nonamenable, by Proposition~\ref{or-cohom} we have 
$H^1(H,\ell^\Phi(H))=0$. Since $N$ is infinite, $\ell^\Phi(G)^N=0$. 
Thus, by Proposition~\ref{rest}(2), the restriction map 
$\mathrm{Rest}_G^N:H^1(G,\ell^\Phi(G))\to H^1(N,\ell^\Phi(G)|_N)$
considered as acting into $H^1(N,\ell^\Phi(G)|_N)$ is injective. In the sequence
$$
H^1(G,\ell^\Phi(G)) \overset{\mathrm{Rest}_G^H}{\longrightarrow} 
H^1(H,\ell^\Phi(G)|_H) \overset{\mathrm{Rest}_H^N}{\longrightarrow} 
H^1(N,\ell^\Phi(G)|_N) ,
$$  
the composition $\mathrm{Rest}_G^N = \mathrm{Rest}_G^H \circ \mathrm{Rest}_H^N$
is injective but $\mathrm{Rest}_G^H = 0$ because $H^1(H,\ell^\Phi(G)|_H)=0$ 
by Lemma~\ref{cohom-rest}. Hence, $H^1(G,\ell^\Phi(G))=0$. 
\end{proof}

\begin{corollary}
Under the conditions of Theorem~\ref{inf-norm}, $\overline{H}^1(G,\ell^\Phi(G))=~0$. 
\end{corollary}


\section{$\Phi$-Harmonic Functions}\label{phi-harm}

Let $G$ be a finitely generated group with finite generating set $S$, 
and suppose that $G$ acts on a countable set $X$. 

If $A$ is an abelian group then denote by $A^X$ the abelian group of 
all functions $f:X\to A$. Denote by $\lambda_{X}:G\to A^X$
the {\it permutation representation} of $G$ on $A^X$:
$$
\lambda_X(g) f(x) = f(g^{-1} x), \quad f\in A^X, \,\, g\in G.
$$
This turns $A^X$ into a $G$-module. If $X=G$ this representation is called
the {\it left regular representation of} of~$G$ in~$A^G$.

The following assertion is well known (see, for example, \cite[Lemma~2.1]{MV07}):

\begin{proposition}\label{cohom-func}
Suppose that a (discrete) group~$G$ acts freely on a set~$X$. Then
$H^1(G,A^X)=0$.
\end{proposition}

Below we write $\mathcal{F}(X)$ instead of $\mathbb{R}^X$.

Introduce the space of {\it $\Phi$-Dirichlet finite functions} 
\begin{equation} 
\begin{split}
D^{\Phi}(X) &= \{f \in \mathcal{F}(X) \mid\| \lambda_{X}(g)f - f\|_{\ell^{(\Phi)}(X)} < \infty \mbox{ for all }g \in G\} \\
&=\{f \in \mathcal{F}(X) \mid\| \lambda_{X}(s)f - f\|_{\ell^{(\Phi)}(X)} < \infty \mbox{ for all }s \in S\}. 
\end{split}
\notag
\end{equation}
Let $D^{\Phi}(X)^{G}$ be the space of functions in~$D^\Phi(X)$ constant on the $G$-orbits of~$X$ 
and let $\ell^{\Phi}(X)^G$ be the space of $\ell^{\Phi}$-functions constant on $G$-orbits. Endow
$\mathscr{D}^{\Phi}(X) = D^{\Phi}(X)/D^{\Phi}(X)^{G}$ with the norm
$$
\|f\|_{ \mathscr{D}^{\Phi}(X)} = \sum_{s \in S}\|\lambda_{X}(s)f - f\|_{\ell^{(\Phi)}(X)}.
$$
Clearly, $D^{\Phi}(X)^{G}$ is the kernel of the map
$$
\|\cdot\|_{ \mathscr{D}^{\Phi}(X)}: D^{\Phi}(X) \to \mathbb{R}.
$$ 
Define the linear mapping $\alpha : D^{\Phi}(X) \to Z^{1}(G,~\ell^{\Phi}(X))$ 
by setting
$$
\alpha(f)(g) = \lambda_{X}(g)f - f.
$$
The mapping induced by~$\alpha$ on $\mathscr{D}^{\Phi}(X)$ 
is an injection since $D^{\Phi}(X)^{G}$ is the kernel of $\alpha$.

Put $\ell_G^{\Phi}(X) = \ell^{\Phi}(X)/\ell^{\Phi}(X)^{G}$.

\begin{theorem}\label{top_isomor}
Suppose that a finitely generated group $G$ acts freely on a countable set~$X$. 
Then $\alpha : \mathscr{D}^{\Phi}(X) \to Z^{1}(G, \ell^{\Phi}(X))$  is a topological isomorphism, 
which implies the following:
\begin{enumerate}
\item $H^{1}(G, \ell^{\Phi}(X)) \cong \mathscr{D}^{\Phi}(X)/\ell_G^{\Phi}(X)$
\item $\overline{H}^{1}(G, \ell^{\Phi}(X)) \cong \mathscr{D}^{\Phi}(X)/\overline{\ell_G^{\Phi}(X)}$
\end{enumerate}
\end{theorem} 

\begin{proof}
From the previous considerations, the map $\alpha$ is continuous and injective. Prove that 
it is surjective. By Proposition~\ref{cohom-func}, $H^{1}(G, \mathcal{F}(X)) = 0$. Hence, 
for every $b \in Z^{1}(G, \ell^{\Phi}(X))$ there exists a function $f \in \mathcal{F}(X)$ such
that $b(g) = \lambda_{X}(g)f - f$ for all $g \in G$. Clearly, $f \in D^{\Phi}(X)$. 
Therefore, $\alpha (f) = b$. It is easy to see that $\alpha ^{-1}$ is continuous.
The definition of~$\alpha$ yields $\alpha(\ell_G^{\Phi}(X)) = B^{1}(G, \ell^{\Phi}(X))$. 
The theorem is proved.
\end{proof}

From now on and until the end of the section, we will assume that $\Phi$ 
is {\it continuously differentiable}.

\begin{definition}
Suppose that $G$ is a finitely generated group, $S$ is its finite set of generators, and
$G$ acts on a countable set $X$. Define the 
Define the {\it $\Phi$-Laplacian} $\Delta_{\Phi} : \mathcal{F}(X) \to \mathcal{F}(X)$
as follows:  
$$
(\Delta_{\Phi}f)(x) := \sum_{s \in S}\Phi'(f(s^{-1}x) - f(x)) 
\text{ for $f \in \mathcal{F}(X)$ and $x \in X$}.
$$
A function $f \in D^\Phi(X)$ is called {\it $\Phi$-harmonic} if $(\Delta_{\Phi}f)(x) = 0$ 
for all $x \in X$. Denote the set of $\Phi$-harmonic functions on~$X$ by $HD^{\Phi}(X)$.
\end{definition}
                                                                        
Introduce a pairing 
$\langle \Delta_{\Phi}*, \, *\rangle : D^{\Phi}(X) \times D^{\Phi}(X) \to \mathbb{R}$
by the formula
$$
\langle \Delta_{\Phi}h, \, f\rangle := \sum_{x \in X}\sum_{s \in S}\Phi'(h(s^{-1}x) - h(x))(f(s^{-1}x) - f(x))
$$
This sum is finite by Proposition~\ref{orl_class}. 

Note that the $\Phi$-Laplacian and the form $\langle \Delta_\Phi\cdot, \cdot\rangle$ are 
well defined on the elements of $\mathscr{D}^\Phi(X)$. Preserve their notation for this space.

Recall (see~\cite{Mus83}) that if $V$ is a real vector space then the functional  
$\rho: V\to [0,\infty]$ is called a {\em modular} on~$V$ if the following hold for all $x,y\in V$:
\begin{enumerate}
\item $\rho(0)=0$;
\item $\rho(-x)=\rho(x)$;
\item $\rho(\alpha x+\beta y)\leq \rho(x) + \rho(y)$ for $\alpha,\beta\geqslant 0$,
$\alpha+\beta=1$;
\item $\rho(x)=0$ implies $x=0$.
\end{enumerate}
The space $\mathscr{D}^{\Phi}(X)$ is endowed with the modular $\rho: \mathscr{D}^{\Phi}(X) \to \mathbb{R}^{+}$,
$$
\rho(f) =\sum_{s \in S} \sum_{x \in X}\Phi\bigl(f(s^{-1}x) - f(x)\bigr).
$$
The \emph{G\^ateaux differential} $\rho'_f$ of the mapping~$\rho$ 
at a point $f \in \mathscr{D}^{\Phi}(X)$ is defined as
$$
\rho_{f}'(g) =\lim_{t \to 0^{+}} \frac{\rho(f+tg) - \rho(f)}{t}.
$$
It is easy to check that $\rho_{f}'(g) = \langle \Delta_{\Phi}f, \, g\rangle$. 

\begin{proposition}\label{dist}
Assume that $\Phi$ is a continuously differentiable strictly convex $N$-function. Let $f_1, \, f_2 \in  D^{\Phi}(X)$. Then $f_1 - f_2 \in D^{\Phi}(X)^{G}$ if and only if
$$
\langle \Delta_{\Phi}f_1, \, f_1 - f_2\rangle = \langle \Delta_{\Phi}f_2, \, f_1 - f_2\rangle
$$
\end{proposition}

\begin{proof}
Let $f_1 - f_2 \in D^{\Phi}(X)^{G}$. It is not hard to see that $\langle \Delta_{\Phi}f, \, *\rangle$ 
maps $D^{\Phi}(X)^{G}$ to $\{0\}$ for all $f \in D^{\Phi}(X)$, and so  
$\langle \Delta_{\Phi}f_1, \, f_1 - f_2\rangle = 0 = \langle \Delta_{\Phi}f_2, \, f_1 - f_2\rangle$.
Conversely, suppose that $f_1- f_2 \notin D^{\Phi}(X)^{G}$. Then $f_1$ and $f_2$ define different elements in  
$D^{\Phi}(X)/D^{\Phi}(X)^{G}$. Involving \cite[Proposition~5.4, p.~24]{EkTem}, we conclude that
$$
\rho(f_1) > \rho(f_2) + \rho_{f_2}'(f_1 - f_2) = \rho(f_2) + \langle \Delta_{\Phi}f_2, \, f_1 - f_2\rangle.
$$
Repeating the same argument, we have
$$
\rho(f_2) > \rho(f_1) + \rho_{f_1}'(f_2 - f_1) = \rho(f_1) - \langle \Delta_{\Phi}f_1, \, f_1 - f_2\rangle
$$
Thus, $\langle \Delta_{\Phi}f_1, \, f_1 - f_2\rangle > \langle \Delta_{\Phi}f_2, \, f_1 - f_2\rangle$.
\end{proof}

For every $x \in X$, define a function $\delta_{x}: X \to \mathbb{R}$ by
$$
\delta_{x}(t) = 
\begin{cases}
1 &\text{if $t = x$}\\
0 &\text{if $t \ne x$}
\end{cases}
$$ 

\begin{lemma}\label{harm_cond}
The following are equivalent for $h \in D^{\Phi}(X)$:
\begin{enumerate}
\item $h \in HD^{\Phi}(X)$;
\item $\langle \Delta_{\Phi}h,\, \delta_x\rangle = 0$ for all $ x \in X$;
\item $\langle \Delta_{\Phi}h,\,f\rangle = 0$ for all $f \in (\overline{\ell_G^{\Phi}(X)})_{\mathscr{D}^\Phi(X)}$.
\end{enumerate}
\end{lemma} 

\begin{proof}
(1)$\Leftrightarrow$(2):
We have
\begin{align*} 
\langle \Delta_{\Phi}h,\, \delta_x\rangle &= \sum_{t \in X}\sum_{s \in S}\Phi'\bigl(h(s^{-1}t) - h(t)\bigr)\bigl(\delta_x(s^{-1}t) - \delta_x(t)\bigr) \\
&=\sum_{s \in S}\bigl( \Phi'(h(x) - h(sx)) - \Phi'(h(s^{-1}x) - h(x))\bigr) \\
&=-\sum_{s \in S} \Phi'(h(sx) - h(x)) -\sum_{s \in S} \Phi'(h(s^{-1}x) - h(x))\\
&= -2\sum_{s \in S} \Phi'(h(s^{-1}x) - h(x)) = -2\Delta_{\Phi}h(x).
\end{align*}
This implies that $\Delta_{\Phi}h(x) = 0$ for all $x \in X$ if and only if 
$\langle \Delta_{\Phi}h,\, \delta_x\rangle = 0$ for all $x \in X$.

The implication (3)$\Rightarrow$(2) is trivial.

(1)$\Rightarrow$(3): 
Since $(\overline{\ell_G^{\Phi}(X)})_{\mathscr{D}^{\Phi}(X)}$ coincides with the closure of the linear span $\Delta$ of the set $\{\, \delta_x \mid x\in X\, \}$ â $\mathscr{D}^{\Phi}(X)$, for every 
$f \in \ell_G^{\Phi}(X)$ there exists a sequence $\{f_n\} \subset \Delta$ such that
 $\|f - f_n\|_{\mathscr{D}^{\Phi}(X)} \to 0$ as $n \to \infty$. Suppose now that $h \in {HD}^{\Phi}(X)$ and  
$m = \max\limits_{s \in S}\sum_{x \in X}\Psi\Bigl(\Phi'\bigl(h(s^{-1}x) - h(x)\bigr)\Bigr)$. 
Then $m < \infty$ by Proposition~\ref{orl_class}, and hence 
$\Phi'(\lambda_{X}(s)h - h)/m \in \tilde\ell^{\Psi}_{1}(X)$ for all $s$. Thus, we infer
\begin{multline*}
0 \le |\langle\Delta_{\Phi}h,\,f\rangle| = |\langle\Delta_{\Phi}h,\,f - f_n\rangle| \\
=\left|m\sum_{s \in S}\sum_{x \in X}\frac{\Phi'((\lambda_{X}(s)h - h)(x))}{m}(\lambda_{X}(s)(f - f_n)-(f - f_n))(x)\right|\\
\le |m| \sum_{s \in S}\|(\lambda_{X}(s)(f - f_n)-(f - f_n))\|_{\ell^{\Phi}(X)} \\
\le 2|m| \sum_{s \in S}\|(\lambda_{X}(s)(f - f_n)-(f - f_n))\|_{\ell^{(\Phi)}(X)} 
= 2|m|\|f - f_n\|_{\mathscr{D}^{\Phi}(X)} \to 0\text{ as } n \to \infty.
\end{multline*}
\end{proof}
\begin{theorem}\label{decomp}
Suppose that $\Phi$ is a continuously differentiable strictly convex $N$-function belonging to~$\Delta_2 (0)\cap \nabla_2(0)$. Let $G$ be a finitely generated group acting on  a countable set~$X$.
Then for every $f \in D^{\Phi}(X)$ there exists a decomposition $f = u + h$, where $u \in (\overline{\ell^{\Phi}(X)})_{D^{\Phi}(X)}$ and 
$h \in HD^{\Phi}(X)$. 
It is unique up to an element of $D^{\Phi}(X)^{G}$.
\end{theorem}

\begin{proof}
Since $\Phi \in \Delta_2(0) \cap \nabla_2(0)$, the space $\ell^{\Phi}(X)$ is reflexive (see \cite{RaRen02}). The same holds for $\mathscr{D}^\Phi(X)$.

Fix a finite generating set $S$ in $G$.

Let $f\in D^\Phi(X)$; we will denote the corresponding element of $\mathscr{D}^\Phi(X)=D^\Phi(X)/D^\Phi(X)^G$ by the same symbol~$f$.   

Put 
$$
d = \inf_{g \in (\overline{\ell_G^{\Phi}(X)})_{\mathscr{D}^{\Phi}(X)} }\rho(f - g),
$$
$$
B = \{g \in (\overline{\ell^{\Phi}_G(X)})_{\mathscr{D}^{\Phi}(X)}\mid \rho(f - g) \le d + 1\}.
$$
Obviosuly, $B$ is bounded since, by the properties of Orlicz spaces, for all $f \in \mathscr{D}^{\Phi}(X)$ we have $\|f\|_{\mathscr{D}^{\Phi}(X)} \le \rho(f) + k$, where $k$ is the cardinality of~$S$. Now, prove that $B$ is closed. Suppose that
$\|f - f_n\|_{\mathscr{D}^{\Phi}(X)} \to 0$ as $n \to \infty$, 
where $f$,\,$f_n \in (\overline{\ell_G^{\Phi}(X)})_{\mathscr{D}^{\Phi}(X)}$.
Since $\Phi$ is $\Delta_2$-regular, convergence in~$\ell^\Phi(X)$ 
in norm is equivalent to convergence in the modular~$\rho_\Phi$ (see \cite[Theorem~9.4]{KraRu}).
Consequently, in $\mathscr{D}^{\Phi}(X)$, convergence in the norm
$\|\cdot\|_{\mathscr{D}^\Phi(X)}$ is equivalent to convergence in~ $\rho$. 
We may thus apply \cite[Corollary~15, p.~86]{RaRen91}, by which 
$\rho(g_n) \to 0$ as $n \to \infty$ implies $\rho(g+g_n) \to \rho(g)$ 
as $n \to \infty$. Consequently,
$$
\rho(f_n) = \rho((f_n - f) + f) \to \rho(f),
$$ 
and thus the condition $f_n \in B$ for all $n$ implies that $f \in B$. Consequently,
 $B$ is a bounded closed convex subset in the reflexive Banach space $\mathscr{D}^\Phi(X)$. 
Hence, as follows from Kakutani's Theorem (see, for example, \cite[Corollary 10.6.2]{Kut}), 
$B$ is compact in the weak topology. Therefore, the weakly lower semi-continuous functional 
$$
F(g) = \rho(f - g),~g \in (\overline{\ell_G^{\Phi}(X)})_{\mathscr{D}^{\Phi}(X)}
$$
attains its minimum~$d$ on $B$. Let $F(u) = d$ and $h = f - u$. 
For $v \in \ell_G^{\Phi}(X)$, consider the smooth function 
$$
F_{v}(t) = \rho(f - (u - tv)), \quad t \in\mathbb{R}.
$$ 
Obviously, the minimum of~$F$ is attained for $t = 0$, which means that
$$
\left.\frac{dF_{v}(t)}{dt}\right|_{t = 0} 
= \rho'_{h}(v)=\langle \Delta_{\Phi}h, \, v\rangle =0 \,\mbox{ for all } v \in \ell_G^{\Phi}(X).
$$
Therefore, $\langle \Delta_{\Phi}h, \, \delta_x\rangle =0$ for all $x \in X$, and, consequently, 
$ h\in HD^{\Phi}(X)$ by Lemma~\ref{harm_cond}.

Prove the uniqueness. Suppose that $f = u_1+h_1 = u_2 + h_2$. 
Appealing to Lemma~\ref{harm_cond}, we conclude that 
$\langle \Delta_{\Phi}h_1, \, h_1 - h_2\rangle = \langle \Delta_{\Phi}h_1, \, u_1 - u_2\rangle =0$,
and, similarly, $\langle \Delta_{\Phi}h_2, \, h_1 - h_2\rangle = 0$. 
By Proposition~\ref{dist}, we have  $h_1 - h_2 = u_1 - u_2 \in D^{\Phi}(X)^{G}$. 
\end{proof}

\begin{corollary}\label{harm+decomp}
If the~action of~$G$ on~$X$ is free then $\overline{H}^{1}(G,\,\ell^{\Phi}(X))$ is 
naturally identified with $HD^{\Phi}(X)/D^{\Phi}(X)^{G}$.
\end{corollary}

{\it Proof} is obtained by analyzing the proofs of Theorem~\ref{top_isomor} and Theorem~\ref{decomp}.  $\square$


\section{Normal Subgroups with Large Centralizer}\label{central}

\begin{theorem}\label{large_centr}
Let $\Phi$ be an $N$-function in $\Delta_2(0)\cap \nabla_2(0)$ and let  
$N$ be an infinite finitely generated normal subgroup of a finitely generated group~$G$.
If $N$ is non-amenable and its centralizer $Z_{G}(N)$ is infinite then 
$\overline{H}^{1}(G,\ell^{\Phi}(G)) = 0$.
\end{theorem}

\begin{proof} 
It is known that every Orlicz space $\ell^{\Phi}(G)$ is topologically isomorphic 
to the space $\ell^{\tilde{\Phi}}(G)$, where $\tilde{\Phi}$ is a continuously 
differentiable strictly convex function $N$-function such that 
\begin{equation}\label{equiv-func}
\Phi(u)\le \tilde{\Phi}(u)\le 2\Phi(u)
\end{equation}
(see \cite{KraRu,Rao68}). Thus, we have topological isomorphisms
$H^1(G, \ell^{\Phi}(G)) \cong H^{1}(G, \ell^{\tilde{\Phi}}(G))$ and 
$\overline{H}^1(G, \ell^{\Phi}(G)) \cong \overline{H}^{1}(G, \ell^{\tilde{\Phi}}(G))$. 
Inequalities~(\ref{equiv-func}) imply that $\Phi$ and~$\tilde\Phi$
are equivalent $N$-functions \cite{KraRu,RaRen91}. Hence, if $\Phi\in\Delta_2(0)$ 
then $\tilde\Phi\in\Delta_2(0)$ (see, for example, \cite[p.~33]{RaRen91});  
moreover, for the functions $\Psi$ and $\tilde\Psi$ complementary to~$\Phi$ 
and~$\tilde\Phi$ respectively, from~(\ref{equiv-func}) we have 
$$
\Psi\left(\frac{v}{2}\right)\le \tilde{\Psi}(v)\le \Psi(v).
$$
This implies that $\nabla_2$-regularity is also preserved in passing 
from~$\Phi$ to $\tilde\Phi$ and hence $\ell^{\tilde\Phi}$ is reflexive.
Consequently, we may assume  without loss of generality that $\Phi$ itself 
is a continuously differentiable strictly convex $N$-function. 

Observe first that $\mathscr{D}_{\Phi}(G)^{Z_{G}(N)} = 0$. 
Indeed, let $[f] \in \mathscr{D}_{\Phi}(G)^{Z_{G}(N)}$ be the equivalence class 
of $f$. Then $\lambda_{G}(g)f - f \in D^\Phi(G)^{N}$ for all $g \in Z_{G}(N)$; 
in other words, $\lambda_{G}(n)( \lambda_{G}(g)f -f) = \lambda_{G}(g)f - f$ 
for all $n\in N$. This is equivalent to 
$\lambda_{G}(g)( \lambda_{G}(n)f -f) = \lambda_{G}(n)f - f$ because 
$n,\,g \in Z_{G}(N)$. Since $[f] \in \mathscr{D}_{\Phi}(G)^{Z_{G}(N)}$, we have 
$\lambda_{G}(n)f -f \in \ell^{\Phi}(G)$, which implies that 
$ \lambda_{G}(n)f -f \in \ell^{\Phi}(G)^{Z_{G}(N)}$. Recalling that 
$Z_{G}(N)$ is infinite, we conclude that $\lambda_{G}(n)f -f$ vanishes 
idenically; i.e., $f \in D^{\Phi}(G)^{N}$. Thus, $[f]= 0$. 

Return to the proof of the theorem. Let 
$\alpha: \mathscr{D}^{\Phi}(G) \to Z^{1}(N, \ell^{\Phi}(G)|_{N})$ 
be the same as in Theorem~\ref{top_isomor}. As was mentioned above, the action 
of~$G$ on $\mathscr{D}^{\Phi}(G)$ induces an action on 
$Z^{1}(N, \ell^{\Phi}(G)|_{N})$, which is defined by 
$$
(g\cdot b)(n) = \lambda(g)(b(g^{-1}ng)), \text{where 
$b \in Z^{1}(N, \ell^{\Phi}(G)|_{N})$, $g \in G$, $n \in N$}.
$$ 
Since this action leaves $B^{1}(N, \ell^{\Phi}(G)|_{N})$ invariant,
it is well defined on $H^{1}(N, \ell^{\Phi}(G)|_{N})$. By computation, we obtain 
$$
(m\cdot b)(n) = b(n) + (\lambda(n)b(m) - b(m)) \mbox{ for } m \in N.
$$
In view of the above, the action of $N$ on $H^{1}(N, \ell^{\Phi}(G)|_{N})$ 
is trivial. It is easy to verify that $\alpha$ is $G$-equivariant. Hence, 
by Corollary~\ref{harm+decomp}, there exists a  
$Z_{G}(N)$-equivariant bijection between $HD^{\Phi}(G)/{D}^{\Phi}(G)^N$ 
and 
$\overline{H}^{1}(N, \ell^{\Phi}(G)|_{N})$. Now, there are two cases depending 
on whether $N \cap Z_{G}(N)$ is infinite or finite. 

First, assume that it is infinite. The induced action of~$N$ leaves the elements
 of $HD^{\Phi}(G)/D^{\Phi}(G)^N$ fixed since it is trivial. Repeating the above
argument, we have $\mathscr{D}_{\Phi}(G)^{N \cap Z_{G}(N)} = 0$, and hence 
$HD^{\Phi}(G)/D^{\Phi}(G)^N = 0$. It follows that 
$\overline{H}^{1}(N, \ell^{\Phi}(G)|_{N}) =~0$, and, by Lemma~\ref{cohom-rest}, 
$\overline{H}^{1}(G, \ell^{\Phi}(G)) = 0$.

Now, suppose that the intersection is finite. For $N$ non-amenable, 
we have $\overline{H}^{1}(N, \ell^{\Phi}(G)|_{N}) = H^{1}(N, \ell^{\Phi}(G)|_{N})$. 
Similarly to previous situations, we conclude that  
$H^{1}(N, \ell^{\Phi}(G)|_{N})^{Z_{G}(N) / (N \cap Z_{G}(N))}= 0$, and so  
$H^{1}(N, \ell^{\Phi}(G)|_{N})^{G/N} = 0$. This and Lemma~\ref{cohom-rest} 
imply that $H^{1}(G, \ell^{\Phi}(G)) = 0 = \overline{H}^{1}(G, \ell^{\Phi}(G))$.
\end{proof}

The same argument that in the proof of Theorem~\ref{large_centr} 
yields the following

\begin{corollary}\label{inf-center}
If $\Phi\in \Delta_2(0)\cap\nabla_2(0)$ and a finitely generated group $G$ has 
infinite center then $\overline{H}^{1}(G, \ell^{\Phi}(G)) = 0$.   
\end{corollary}

As in~\cite{MV07}, we find a sufficient condition for the triviality of the first
$\ell^\Phi$-cohomology of the wreath product of two countable discrete groups.

\begin{theorem}\label{wreath} 
Suppose that $G_1$, $G_2$ are nontrivial countable discrete groups, 
$\Phi\in\Delta_2(0)\cap\nabla_2(0)$, and $G=G_1 \wr G_2$. If $G_1$ is nonamenable then 
$H^{1}(G, \ell^{\Phi}(G)) = 0$.  
\end{theorem}

\begin{proof}
Put $N=G_1 \times (\oplus_{G_2} G_1)$. Then $N$ is nonamenable. Represent $N$ as
$N=G_1 \times G_0$, where $G_0= \oplus_{G_2\setminus\{e\}} G_1$.  
By~\cite[Theorem~2, pp.~297--298]{RaRen91}, we can replace our $N$-function 
$\Phi\in \Delta_2(0)\cap \nabla_2(0)$ by an equivalent $N$-function 
$\Phi_1\in \Delta_2(0)\cap\nabla_2(0)$ such that the spaces $\ell^\Phi(N)$
and $\ell^{\Phi_1}(N)$ are isomorphic and $\ell^{\Phi_1}(N)$ is uniformly convex.
Of course, $\ell^{\Phi_1}(N)$ is also reflexive. The replacement of $\Phi$
by $\Phi_1$ does not change cohomology, and so we will assume without 
loss of generality that $\ell^\Phi(N)$ is uniformly convex. 
The left regular representation $\lambda_N$ of $N$ on $\ell^\Phi(N)$ is such that
$\lambda_N|_{G_1}$ and $\lambda_N|_{G_0}$ do not have invariant vectors (because
$G_1$ is infinite) and $\lambda_N$ does not have almost invariant vectors 
(by the nonamenability of~$N$; see \cite[Proposition~2, pp.~387--389]{Ra04a}).
Suppose that 
\begin{equation}\label{cohom-nontr}
H^{1}(N, \ell^{\Phi}(N)) \ne 0.      
\end{equation} 
As is well known (see Lemmas~2.2.1 and~2.2.6 in \cite{BeHaV08}, which are formulated for actions
on Hilbert modules but the proofs also hold for Banach modules), 
if $G$ is a topological group and $V$ is a $G$-module
then there is a correspondence between $Z^1(G,V)$ and continuous actions of~$G$ by 
affine isometries on~$V$; moreover, the affine action corresponding to a cocycle has 
fixed points if and only if it is a coboundary. Therefore, relation~(\ref{cohom-nontr}) implies 
the existence of an action of~$N$ by affine isometries on~$\ell^\Phi(N)$ without fixed points. 
Then, by \cite[Theorem~7.1]{BFGM}, since $\ell^\Phi(N)$ is a uniformly convex $N$-module 
and the left regular representation of~$N$ on~$\ell^\Phi(N)$
does not have almost invariant vectors, there must exist a nonzero invariant vector for
$\lambda_N|_{G_1}$ and $\lambda_N|_{G_0}$. This is a~contradiction to what was said above.
Thus, (\ref{cohom-nontr}) fails and $H^{1}(N, \ell^{\Phi}(N)) = 0$. Consequently,
by Proposition~\ref{cohom-rest}, $H^{1}(N, \ell^{\Phi}(G)|_N) = 0$. Then, 
by Proposition~\ref{rest}, $H^{1}(G, \ell^{\Phi}(G)) = 0$.
   
The theorem is proved.
\end{proof}

\smallskip
The authors are grateful to Prof. Anna Kami{\'n}ska for sending her papers about discrete
Orlicz spaces and to Dr.~Viorica Montreanu for providing them with a collection of articles 
on the $\Phi$-Laplacian in the ``continuous'' case.

\end{document}